\documentclass[12pt, reqno]{amsart}
\usepackage{amsmath, amsthm, amscd, amsfonts, amssymb, graphicx, color}
\usepackage[bookmarksnumbered, colorlinks, plainpages]{hyperref}

\hypersetup{colorlinks=true,linkcolor=red, anchorcolor=green, citecolor=cyan, urlcolor=red, filecolor=magenta, pdftoolbar=true}

\newtheorem{theorem}{Theorem}[section]
\newtheorem{lemma}[theorem]{Lemma}
\newtheorem{proposition}[theorem]{Proposition}
\newtheorem{corollary}[theorem]{Corollary}
\theoremstyle{definition}

\newtheorem{example}[theorem]{Example}

\theoremstyle{remark}
\newtheorem{remark}[theorem]{Remark}

\numberwithin{equation}{section}

\def\bC{\Bbb C}
\def\bK{\Bbb K}

\def\bM{\Bbb M}

\def\ep{\epsilon}
\def\op{\overline p}
\def\oq{\overline q}
\def\ou{\overline u}

\def\ep{\epsilon}

\def\b1{\bold 1}
\def\tA{\widetilde A}

\def\tB{\widetilde B}

\begin{document}

\setcounter{page}{1}

\title[Strongly Operator Convex Functions]{On the Property $IR$ of Friis and R\o rdam}

\author[Lawrence G.~Brown]{Lawrence G.~Brown}

\address{Department of Mathematics, Purdue University, 150 N.~University Street, West Lafayette, IN\ \ 47907}
\email{\textcolor[rgb]{0.00,0.00,0.84}{lgb@math.purdue.edu}}

\dedicatory{Dedicated to the memory of Ronald G.~Douglas}




\begin{abstract}
In a 1997 paper Lin solved a longstanding problem as follows: For each $\ep>0$, there is $\delta>0$ such that if $h$ and $k$ are self-adjoint contractive $n\times n$ matrices and $\| hk-kh\|<\delta$, then there are commuting self-adjoint matrices $h'$ and $k'$ such that $\|h'-h\|$, $\|k'-k\|<\ep$. Hence $\delta$ depends only on $\ep$ and not on $n$. In a 1996 paper Friis and R\o rdam greatly simplified Lin's proof by using a property they called $IR$. They also generalized Lin's result by showing that the matrix algebras can be replaced by any $C^*$-algebras satisfying $IR$. The purpose of this paper is to study the property $IR$. One of our results shows how $IR$ behaves for $C^*$-algebra extensions. Other results concern non-stable $K-$theory.  One shows that $IR$ (at least the stable version) implies a cancellation property for projections which is intermediate between the strong cancellation satisfied by $C^*$-algebras of stable rank one, and the weak cancellation defined in a 2014 paper of Pedersen and the author.
\end{abstract} \maketitle

\section{\textbf{Definitions and basic results}}

For a non-unital $C^*$--algebra $A$, $\tA$ denotes the result of adjoining an identity, and $\tA=A$ if $A$ is unital. The identity of any unital $C^*$--algebra is denoted by $\b1$. Also, for any $C^*$--algebra $A$, we denote by $A^+$ the result of adjoining a new identity. Thus $A^+=\tA$ if $A$ is non-unital and $A^+\cong A\oplus\bC$ if $A$ is unital. For $a\in A$ its spectrum is denoted by $\sigma(a)$. The set of invertible elements of a unital $C^*$--algebra $A$ is denoted by $GL(A)$.

For a unital $C^*$--algebra $A$, Friis and R\o rdam \cite{FR} defined $R(A)$ as follows: The element $a$ is in $R(A)$ if and only if there does not exist a (closed two--sided) ideal $I$ such that $\pi(a)$ is one--sided invertible but not invertible, where $\pi:A\to A/I$ is the quotient map. They then said that $A$ satisfies $IR$ if $R(A)$ is in the (norm) closure of $GL(A)$. Of course it is obvious that ${\overline{GL(A)}} \subset R(A)$, since no element that is one--sided invertible but not invertible (in some quotient algebra of $A$) is approximable by invertibles.  For non-unital $A$, they said $A$ satisfies $IR$ if $\tA$ does. For expository purposes we introduce a formally weaker property, but in Proposition 2.2 we will show it is equivalent. An arbitrary $C^*$--algebra $A$ satisfies $IR_0$ if every element of $R(A^+)\cap (\b1+A)$ is in the closure of $GL(A^+)$. Obviously $IR_0$ is equivalent to $IR$ in the unital case, but for non-unital $A$ we are leaving open (for now) the possibility that $R(\tA)\cap A$ is not contained in the closure of $GL(\tA)$. 

It is convenient to introduce some hereditary $C^*$--subalgebras of the unital algebra $A$. For $t\in A$ and $\ep >0$, $B_{r,\ep}(t)$ denotes the hereditary $C^*$--algebra whose open projection is $\chi_{[o,\ep)}(|t|)$ and $B_{\ell,\ep}(t)$ denotes the hereditary $C^*$--algebra whose open projection is $\chi_{[0,\ep)}(|t^*|)$. Here $\chi_{[o,\ep)}(|t|)$, for example, is a spectral projection of $|t|$ in the enveloping von Neumann algebra $A^{**}$. It is not necessary for the reader to know about open projections to understand these concepts. Thus $B_{r, \ep}(t)$ is just the hereditary $C^*$--algebra generated by $f(|t|)$, where $f$ is a continuous function such that $\{x:f(x)=0\}=[\ep,\infty)$. If $I_+$  and $I_-$ are the ideals generated respectively by $B_{r,\ep}(t)$ and $B_{\ell,\ep}(t)$, then $I_+$ is the smallest ideal such that $\pi(t)$ is left invertible and $m(\pi(t))\ge\ep$. Here, for a left invertible element $s$, $m(s)$ is the largest number such that $\|sv\|\ge m(s)\|v\|$, when the algebra is faithfully represented on a Hilbert space. Also write $m_*(s)=m(s^*)$ where $s$ is right invertible. Then $I_-$ is the smallest ideal such that $\pi(t)$ is right invertible and $m_*(\pi(t))\ge \ep$. Since $\sigma(|t|)\cup\{0\}=\sigma(|t^*|)\cup\{0\}$, then $t\in R(A)$ if and only if $I_+=I_-$ for all $\ep>0$. It is sufficient that this property hold for arbitrarily small values of $\ep$. Also, for example, if for each $\ep>0$ there is $\delta >0$ such that the ideal generated by $B_{r,\ep}(t)$ contains $B_{\ell,\delta}(t)$ and the ideal generated by $B_{\ell,\ep}(t)$ contains $B_{r,\delta}(t)$, then $t$ is in $R(A)$.

For $A$ unital and $\ep>0$ we say that $t$ is {\it $\ep$--almost regular} if $B_{r,\ep}(t)$ and $B_{\ell,\ep}(t)$ generate the same ideal. The same is then true for any $\ep'>\ep$.

\begin{proposition} Let $A$ be a unital $C^*$--algebra and $t\in A$.
\begin{itemize}
\item[(i)] If $t$ is $\ep$--almost regular, then dist$(t,R(A))\le\ep$.
\item[(ii)] If dist$(t,R(A))<\ep$, then $t$ is $2\ep$--almost regular. 
\item[(iii)] If $\| |t|-|t^*|\|<\ep$, then $t$ is $\ep$--almost regular.
\end{itemize}
\end{proposition}

\begin{proof}
\begin{itemize}
\item[(i)] Let $f_{\ep}(x)=\max(x-\ep, 0)$. If $t=v|t|$ is the canonical polar decomposition in $A^{**}$, let $t'=v f_{\ep}(|t|)=f_{\ep}(|t^*|)v$. Then $\|t'-t\|\le \ep$. Clearly, for $\delta >0$, $B_{r,\delta}(t')=B_{r,\delta+\ep}(t)$ and $B_{\ell,\delta}(t')=B_{\ell,\delta+\ep}(t)$. Therefore these hereditary $C^*$--subalgebras generate the same ideal and $t'\in R(A)$.
\item[(ii)] Let $t'$ be an element of $R(A)$ such that $\|t'-t\|=\delta<\ep$. If $t$ is not $2\ep$--almost regular, then there is an ideal $I$ such that, with $\pi: A\to A/I$ the quotient map, either $\pi(t)$ is left invertible but not invertible and $m(\pi(t))\ge  2\ep$ or $\pi(t)$ is right invertible but not invertible and $m_*(\pi(t))\ge 2\ep$. Assume the former without loss of generality. Then $\pi(t')$ is left invertible and $m(\pi(t'))\ge 2\ep-\delta$. Therefore $\pi(t')$ is invertible and $m_*(\pi(t'))=m(\pi(t'))\ge 2\ep-\delta$. Since $2\ep -\delta>\delta$, it follows that $\pi(t)$ is invertible after all. (The proof shows directly that $m_*(\pi(t))\ge 2\ep-2\delta$.)
\item[(iii)] If $t$ is not $\ep$--almost regular, then there is an ideal $I$ such that, with $\pi:A\to A/I$ the quotient map, either $\pi(t)$ is left invertible but not invertible and $m(\pi(t))\ge\ep$ or $\pi(t)$ is right invertible but not invertible and $m_*(\pi(t))\ge\ep$. 
Assume the former and let $\delta=\| |\pi(t)|- |\pi(t^*)| \|<\ep$. 
Since $0\in \sigma(|\pi(t^*)|)$, then dist$(0, \sigma(|\pi(t)|))\le\delta$, a contradiction.
\end{itemize}
\end{proof}

\medskip\noindent
{\it Remark}. Part (i) of \cite [Proposition 4.2]{FR} is that dist$(t,R(A))\le \| |t|-|t^*| \|$. Part (iii) of the above Proposition is obviously inspired by this.

It is of course interesting to find conditions that imply $IR$ and to find how the property $IR$ propagates itself. It was pointed out in \cite{FR} that stable rank one implies $IR$ (Stable rank in this sense was introduced by Rieffel \cite{Ri} and stable rank one means that $GL(\tA)$ is dense in $\tA$.) and  (using \cite{R2}) that unital purely infinite simple $C^*$--algebras have $IR$. The latter will be generalized, using a $K$--theoretic approach, in \S 3 below. Lemma 4.3 of \cite{FR} shows that direct products (also known as $\ell^\infty$--direct sums) of $C^*$--algebras with $IR$ have $IR$, and the proof contains the assertion that $IR$ passes to ideals (see Remark 1.5 below for more on this). We proceed to generalize the latter result. Our proof relies on Proposition 2.2, but the result is included in this section for expository purposes.

\begin{lemma} (cf. \cite[Theorem 3.5]{BP}). Let $B$ be a proper hereditary $C^*$--subalgebra of a unital $C^*$--algebra $A$, and identify $B^+$ with $B+\bC\b1$. For $t$ in $\b1 +B$, if $t\in{\overline{GL(A)}}$, then $t\in{\overline{GL(B^+)}}$. 
\end{lemma}

\begin{proof} The argument is just part of the proof of \cite[Theorem 3.5]{BP} and is included for the convenience of the reader. Let $t=\b1+b$, $b\in B$. Given $\ep>0$ choose $\delta >0$ such that
\[2\delta< 1,\  4\delta \|b\|<1,\  4\delta \|b\|^2<\ep.\] 
Find $a$ in $A$ such that $\b1+a\in GL(A)$ and $\|a-b\|<\delta$. Put $d=\b1-(\b1+a-b)^{-1}$, which is permissible since $\|a-b\|<1$, and let
\begin{align*}
c&= (\b1+a-b)^{-1}(\b1+a)(\b1-db)^{-1}\\
&=(\b1-d)(\b1+a-b+b)(\b1-db)^{-1}\\
&=(\b1 +(\b1-d)b)(\b1-db)^{-1}=\b1+b(\b1-db)^{-1}.
\end{align*}
Note that $\|d\|\le\delta(1-\delta)^{-1}<2\delta$ so that $\|db\|<\frac 12$. 
By construction $c\in GL(A)$. Also
\begin{align*}
\|c-(\b1+b)\|&=\|b(\b1-db)^{-1}-\b1)\|\\
&\le \|b\|^2 \|d\| (1-\|db\|)^{-1}<2\|b\|^2 \|d\|\\
&< 4\delta\|b\|^2<\ep.
\end{align*}
Finally, $c-\b1=b(\b1-db)^{-1}=\sum^\infty_0 b(db)^n\in B$. 
\end{proof}

\begin{lemma} Any hereditary $C^*$--subalgebra of an algebra with $IR_0$ has $IR_0$.
\end{lemma}

\begin{proof} Let $A$ have $IR_0$ and $B$ be a hereditary $C^*$--subalgebra of $A$. Apply Lemma 1.2 with $A^+$ in the role of $A$. Since the image of $B^+$ in any quotient, $A^+/I$, of $A^+$ is a unital $C^*$--subalgebra of $A^+/I$, we see easily that $R(B^+)\subset R(A^+)$. The result is now clear.
\end{proof}

Combining this lemma with Proposition 2.2 below, we have:
\begin{proposition} Any hereditary $C^*$--subalgebra of a $C^*$--algebra with $IR$ also has $IR$.
\end{proposition} 

\begin{remark} The argument given in the proof of \cite [Lemma 4.3]{FR} for the fact that any ideal $I$ in a $C^*$--algebra with $IR$ also has $IR$ actually shows only that $I$ has $IR_0$. Of course this is remedied by our Proposition 2.2, but it is not hard to see, without using any results from the present paper, that \cite [Lemma 4.3]{FR} is correct as stated.
\end{remark}
It is easy to see, as in the proof of Lemma 1.3, that if $B$ is a unital $C^*$--subalgebra of $A$, then $R(B)\subset R(A)$. We have the following partial converse:

\begin{proposition} Let $B$ be a hereditary $C^*$--subalgebra of a unital $C^*$--algebra $A$. Then $(\b1+B)\cap R(A)\subset R(B+\bC\b1)$.
\end{proposition}

\begin{proof} We may assume $B\ne A$. (So $B+\bC\b1$ can be identified with $B^+$. Note that $B$ could be a unital $C^*$--algebra, though not a unital subalgebra of $A$, so $B^+$ may not be the same as $\tB$.) If $t\in (\b1+B)\cap R(A)$ and if $0<\ep<1$, then $B_{r,\ep}(t)$ and $B_{\ell,\ep}(t)$ are both contained in $B$. If these generated distinct ideals of $B$, they would also generate distinct ideals of $A$.
\end{proof}

\begin{example} The last result is not true for elements of $B$; i.e., $B\cap R(A)$ need not be contained in $R(B+\bC\b1)$. To see this let $A_0$ be a non-unital purely infinite simple $C^*$--algebra, let $A=\tA_0$, and let $B=pAp$ for a non-zero projection $p$ in $A_0$. If $u$ is a proper isometry in $B$, then $u\not\in R(B+\bC\b1)$, since $B+\bC\b1\cong B\oplus\bC$, but $u\in R(A)$, since the only relevant quotient of $A$ is $A$ itself.
\end{example}

\begin{lemma} If $A$ is the direct limit of a directed family $\{B_i\}$ of hereditary $C^*$--subalgebras, and if each $B_i$ satisfies $IR_0$, then $A$ satisfies $IR_0$.
\end{lemma}

\begin{proof} Identify each $B_i^+$ with the unital subalgebra $B_i+\bC\b1$ of $A^+$. Let $t\in (\b1+A)\cap R(A^+)$. If $0<\ep<\frac 12$, there are $i_0$ and $t_0\in \b1+B_{i_0}$ such that $\|t_0-t\|<\ep$.
Therefore $t_0$ is $2\ep$--almost regular in $A^+$. 
Let $t_1=vf_{2\ep}(|t_0|)/(1-2\ep)$, where $t_0=v(|t_0|)$ is the canonical polar decomposition of $t_0$ in $(A^+)^{**}$ and $f_{2\ep}$ is as in the proof of Proposition 1.1.
Then $t_1\in R(A_+)\cap (\b1+B_{i_0})$ and $\|t_1-t_0\|\le 2\ep\|x\|/(1-2\ep)$. 
By Proposition 1.6 $t_1\in R(B_{i_0}^+)$. Then $t_1\in{\overline{GL(B_i^+)}}\subset {\overline{GL(A^+)}}$, whence \[\text{dist}(t,GL(A^+))<\ep+2\ep\|x\|/(1-2\ep).\]
Combining this with Proposition 2.2 below we have:
\end{proof}

\begin{proposition} If $A$ is the direct limit of an upward directed family $\{B_i\}$ of hereditary $C^*$--subalgebras, and if each $B_i$ satisfies $IR$, then $A$ satisfies $IR$. 
\end{proposition}

The next result is of interest only in the non-unital case.  The term real rank zero was introduced in \cite{BP1}.  One of many equivalent conditions for real rank zero is that every hereditary $C^*-$subalgebra has an approximate identity consisting of projections.

\begin{corollary} If $A$ has an approximate identity of projections, in particular if $A$ is of real rank zero, then $A$ satisfies $IR$ if and only if $pAp$ satisfies $IR$ for each projection $p$ in $A$.
\end{corollary}

\begin{proof} Let $\{e_i\}$ be an approximate identity of $A$ consisting of projections and let $B_i=e_i A e_i$. Then the result follows from the proof of the proposition. (If $\{e_i\}$ is not increasing, then Proposition 1.9 does not apply directly but its proof still works.)
\end{proof}

\begin{corollary} Any purely infinite simple $C^*$--algebra satisfies $IR$.
\end{corollary}

\begin{proof} The non-unital case follows from the unital case, which is in \cite{FR}, and the above, since purely infinite simple $C^*$--algebras have real rank zero by \cite{Z}.
\end{proof}

It is of course true that the direct sum (also known as the $c_0$--direct sum) of $C^*$--algebras with $IR$ also has $IR$. This follows from \cite [Lemma 4.3]{FR} together with the assertion about ideals in $IR$ $C^*$--algebras contained in the proof of that lemma, or it can be deduced from the statement of \cite [Lemma 4.3]{FR} together with either Proposition 1.4 or Proposition 1.9.

The following proposition will be used in the proof of Theorem 3.2. Although Theorem 3.2 could be proved without it, the proposition may be interesting in its own right.

\begin{proposition} Let $A$ be a unital $C^*$--algebra, $\ep>0$, and $p$ a projection in $A$. If $t$ in $A$ is $\ep$--almost regular, then also $tp$ is $\ep$--almost regular. Also if $t\in R(A)$, then $tp\in R(A)$.
\end{proposition}
\begin{proof} We need to show that if $I$ is an ideal such that $\pi(tp)$ is left invertible with $m(\pi(tp))\ge\ep$ or right invertible with $m_*(\pi(tp))\ge 0$, where $\pi: A\to A/I$ is the quotient map, then $\pi(tp)$ is invertible. In the first case, $\pi(p)$ is 
 also left invertible, whence $\pi(p)=\b1$. Then $\pi(t)$ is left invertible and $m(\pi(t))\ge\ep$. It follows that $\pi(t)$ is invertible, and therefore $\pi(tp)$ is invertible. In the second case, $\pi(t)$ is also right invertible. Since $\|\pi(p)\|\le 1$, we also conclude that $m_*(\pi(t))\ge\ep$. then $\pi(t)$ is invertible, whence $\pi(p)$ is right invertible, whence $\pi(p)=\b1$. 

The second statement follows, since $tp$ $\ep$--almost regular, $\forall\ep>0$, implies $tp\in R(A)$.
\end{proof}

\section{\textbf{Extensions and $IR$}}

\begin{lemma} Let $I$ be an ideal satisfying $IR_0$ in a unital $C^*$--algebra $A$ and $\pi: A\to A/I$ the quotient map. If $\ep>0$, $t$ is an $\ep$--almost regular element of $A$, and $\pi(t)$ is a liftable invertible, then dist$(t,GL(A))\le\ep$.
\end{lemma}
\begin{proof} Let $x$ be an element of $GL(A)$ such that $\pi(xt)=\b1$. Choose $\delta$ such that $0<\delta<\ep$ and $\|\pi(t)^{-1}\|<1/\delta$. Then $B_{\ell,\delta}(t)$ and $B_{r,\delta}(t)$ are contained in $I$. Let $C_\ell=B_{\ell, \ep}(t)\cap I$ and $C_r=B_{r,\ep}(t)\cap I$. Since $B_{\ell, \ep}(t)$ and $B_{r, \ep}(t)$ generate the same ideal $J$, then $C_\ell$ and $C_r$ both generate the ideal $J\cap I$. Now let $y_+=f_1(|t|)$, $z_+=f_2(|t|)$, $y_-=f_1(|t^*|)$, and $z_-=f_2(|t^*|)$ where

\begin{align*}
 f_1(x)=&\left\{
 \begin{aligned}
  1,\quad & 0\le x\le\frac{\delta}3,\\
(\frac{2\delta}3-x)/(\frac{\delta}3),\quad & \frac{\delta}3\le x\le \frac{2\delta}3,\\
0,\quad & x\ge \frac{2\delta}3\\
\end{aligned}\right.
&
 f_2(x)=&\left\{
\begin{aligned} 
  1,\quad & 0\le x\le \frac{2\delta}3,\\
(\delta-x)/(\frac{\delta}3),\quad & \frac{2\delta}3 \le x\le \delta,\\
0,\quad & x\ge \delta.\\
\end{aligned}\right.
\end{align*}

Let $\{e_k\}$ be an increasing approximate identity of $C_r$ such that $e_k\ge z_+$, $\forall k$. This is possible, since we can take $e_k$ of the form $z_++(\b1-z_+)^{\frac 12} f_k(\b1-z_+)^{\frac 12}$. If $t=v|t|=|t^*|v$ is the canonical polar decomposition in $A^{**}$, we can define a similar approximate identity $\{e'_k\}$ of $C_\ell$ by $(\b1-e'_k)=v(\b1-e_k)v^*$. In particular $e'_k\ge z_-$ and $ve_k=e'_kv$. We claim that for $k$ sufficiently large, there is $\eta>0$ such that $y_-$ is in the ideal generated by $r$ for any $r$ with $\|r-e_k\|<\eta$ and $y_+$ is in the ideal generated by $r'$ for any $r'$ with $\|r'-e'_k\|<\eta$. To see the first note that $\|z_--\sum^n_1 a_ic_i b_i \| < 1$ for some $n$, $a_i$,$b_i\in A$, and $c_i\in C_r$. Thus for sufficiently large $k$, $\|z_--\sum^n_1 a_ie_kc_ib_i\|<1$, and the same is true for suitable $\eta$ if $e_k$ is replaced by $r$. Since $z_-y_-=y_-$, any ideal that contains $r$ must contain $y_-$. The same argument applies to the second assertion. Choose one such $k$ and $\eta$ and let $g$ be a continuous function such that $0\le g\le 1$, $g$ vanishes in a neighborhood of 0, $g(1)=1$, $\|g(e_k)-e_k\|<\eta$, and $\|g(e'_k)-e'_k\| < \eta$. Then choose a continuous function $h$ such that $h(0)=0$, $0\le h\le 1$, and $hg=g$. Let $e=h(e_k)$ and $e'=h(e'_k)$. Note that since $g(1)=1$, $1\ge e_k\ge z_+$, and $z_+y_+=y_+$, we have that $g(e_k)y_+=y_+$ and hence $ey_+=y_+$. Similarly $e'y_-=y_-$.

Now consider $s=v(\b1-e)|t|=(\b1-e')v|t|=(\b1-e')t$. We claim that $s\in R(A)$. First suppose $K$ is an ideal such that $s$ is left invertible modulo $K$. then $|t|$ is invertible modulo $K$ and hence $\b1-e$ is (left) invertible modulo $K$. Since $(\b1-e)y_+=0$, $y_+\in K$. Also since $(\b1-e)g(e_k)=0$ and the ideal generated by $g(e_k)$ contains $y_-$, $y_-\in K$. Since $y_\pm\in K$, $t$ is invertible modulo $K$, and hence $\b1-e'$ is left invertible modulo $K$. Since $\b1-e'$ is self--adjoint, $s$ is invertible modulo $K$. A similar argument shows that $s$ right invertible modulo $K$ implies $s$ invertible modulo $K$.

Now let $s_1=xs$. Then $s_1\in R(A)\cap (\b1+I)$, since $x$ is invertible, $xt\in \b1+I$, and $e'\in I$. By Proposition 1.6, $s_1\in R(I+\bC\b1)$. Since $I$ has $IR_0$, $s_1\in {\overline {GL(I+\bC\b1)}}\subset {\overline {GL(A)}}$. Therefore $s\in{\overline{GL(A)}}$. Since $\|e'\|\le 1$ and $e'\in B_{\ell,\ep}(t)$, $\|s-t\|\le \ep$.
\end{proof}

\begin{proposition} The properties $IR$ and $IR_0$ are equivalent.
\end{proposition}

\begin{proof} Let $A$ be a non-unital $C^*$--algebra with $IR_0$. Since $(\tA\setminus A)\cap R(\tA)\subset {\overline{GL(\tA)}}$, it is enough to show $A\cap R(\tA)\subset {\overline{GL(\tA)}}$. Thus let $s\in A\cap R(\tA)$, $\ep>0$, and $t=s+\ep\b1$. Then $t$ is $3\ep$--almost regular, and the image of $t$ modulo $I$ is a liftable invertible. Then Lemma 2.1 implies dist$(t, GL(A))\le 3\ep$. Therefore dist$(s, GL(A))\le 4\ep$. Since $\ep$ is arbitrary, the result follows.
\end{proof}

\begin{corollary} Let $A$ be a unital $C^*$--algebra, $I$ an ideal with $IR$, and $\pi:A\to A/I$ the quotient map. For $t\in A$, $ t\in {\overline{GL(A)}}$ if and only if $t\in R(A)$ and $\pi(t)\in {\overline{\pi(GL(A))}}$.
\end{corollary}

\begin{proof} The necessity is obvious. If the conditions are satisfied, choose $\ep>0$ and $s$ in $GL(A)$ such that $\|\pi(s)-\pi(t)\|<\ep$. Then choose $t'$ in $A$ such that $\|t'-t\|<\ep$ and $\pi(t')=\pi(s)$. Then $t'$ is $2\ep$--almost regular. Thus Lemma 2.1 implies dist$(t', GL(A))\le 2\ep$, whence dist$(t, GL(A))<3\ep$.
\end{proof}

\begin{lemma} If $I$ is an ideal of a unital $C^*$--algebra $A$ and $\ep>0$, then any element of $R(A/I)$ has an $\ep$--almost regular lift.
\end{lemma}

\begin{proof} Let $t$ be an arbitrary lift of an element of $R(A/I)$ and let $\{e_i\}$ be an increasing quasi--central approximate identity of $I$. We claim that $t(\b1-e_i)$ is $\ep$--almost regular for $i$ sufficiently large. First choose $\delta$ with $0<\delta<\ep/2$ and choose $i_0$ such that $i\ge i_0\Rightarrow \|[t,e_i]\|<\delta$. Then choose $\gamma>0$ such that $\gamma\|t\| /(\ep-\delta)\le 1$. Finally choose continuous functions $f_1$,$f_2$ such that $0\le f_i\le 1$, $f_1=1$ in some neighborhood of 0, $f_2f_1=f_1$, and $f_i$ is supported in $[0, \ep/2)$. 

Now let $J_+$ be the ideal generated by $B_{r,\ep/2}(t)$ and $J_-$ the ideal generated by $B_{\ell,\ep/2}(t)$. Since $t$ maps to an element of $R(A/I)$, $J_++I=J_-+I$. Therefore $f_2(|t^*|)\in J_++I$, and we can write $f_2(|t^*|) \in z_-+J_+$ for some positive $z_-$ in $I$. (This follows because the image of $f_2(|t^*|)$ modulo $J_+$ is a positive element of $I/I\cap J_+$.) Similarly write $f_2(|t|)\in z_++J_-$ for some positive element $z_+$ of $I$. Then choose $i_1\ge i_0$ such that for $i\ge i_1$, $\|(\b1-e_i)z_{\pm}\|<\gamma$. 

Fix $i\ge i_1$ and suppose $L$ is an ideal such that $\pi(t(\b1-e_i))$ is left invertible and $m(\pi(t(\b1-e_i)))\ge\ep$, where $\pi:A\to A/L$ is the quotient map. Then since $\|[t, e_i]\|<\delta$, $m(\pi((\b1-e_i)t))>\ep-\delta\ge \ep/2$. Thus $L\supset J_+$. Also $m(\pi(\b1-e_i))\ge \ep / \|t\|$. 
Since $\|(\b1-e_i)z_-\|<\gamma$, it follows that $\|\pi(z_-)\|<\gamma\|t\| / \ep\le 1$. But $\pi(z_-)=\pi(f_2(|t^*|))$ and $\pi(f_2(|t^*|))\pi(f_1(|t^*|))=\pi(f_1(|t^*|))$. Hence $f_1(|t^*|)$ is in $L$, which implies $B_{\ell,\mu}(t)\subset L$ for some $\mu>0$. So $\pi(t)$ is invertible, whence $\pi(t(\b1-e_i))$ is invertible, whence $m_*(\pi(t(\b1-e_i)))\ge \ep$. 

Similarly, we can show that if $\pi(t(\b1-e_i))$ is right invertible and $m_*(\pi(t(\b1-e_i)))\ge\ep$, then $\pi(t(\b1-e_i))$ is invertible and hence $m(\pi(t(\b1-e_i)))\ge\ep$. The small differences are that now we directly see that $m_*(\pi(t))\ge\ep$ and $m_*(\pi(\b1-e_i))\ge(\ep-\delta) / \|t\|$. (The latter since $\|t(\b1-e_i)-(\b1-e_i) t\|<\delta$.) 

The results of the above two paragraphs show that $t(\b1-e_i)$ is $\ep$--almost regular.
\end{proof}

\medskip\noindent
{\it Remark}. The same techniques show that any $\ep_1$--almost regular element of $A/I$ has an $\ep$--almost regular lift if $\ep>\ep_1$. 

\begin{theorem} Let $I$ be a closed two--sided ideal of a $C^*$--algebra $A$. Then $A$ has $IR$ if and only if both $I$ and $A/I$ have $IR$ and invertibles lift from $\tA/I$ to $\tA$.
\end{theorem}

\begin{proof} By replacing $A$ with $\tA$ if necessary, we may assume $A$ is unital. First assume $A$ has $IR$. Then $I$ has $IR$ by Proposition 1.4. If $s\in R(A/I)$, then $s$ can be lifted to an $\ep$--almost regular element $t$ of $A$, by Lemma 2.4, Then dist$(t, GL(A))\le\ep$ by Proposition 1.1 and the $IR$ property for $A$. Therefore dist$(s,GL(A/I))\le \ep$. Since $\ep$ is arbitrary, this implies $A/I$ has $IR$. To show that invertibles lift, we again use Lemma 2.4 (though we could also manage with \cite [Proposition 4.2(ii)]{FR}). If $s\in GL(A/I)$, find a sequence $\{t_n\}$ of lifts such that $t_n$ is $\ep_n$--almost regular and $\ep_n\to 0$. Then find $t'_n\in GL(A)$ such that $\|t'_n-t_n\|\to 0$. Then if $\pi: A\to A/I$ is the quotient map, $\pi(t'_n)\to s$, whence $\pi(t'_n)s^{-1}\to\b1$, whence $\pi(t'_n)s^{-1}$ is a liftable invertible for large $n$. Since $\pi(t'_n)$ is a liftable invertible by construction, then $s$ is liftable.

Now assume that $I$ and $A/I$ have $IR$ and invertibles lift. Let $t\in R(A)$. Then $\pi(t)\in R(A/I)$ (as observed in the proof of \cite [Lemma 4.3]{FR}), since every quotient of $A/I$ is also a quotient of $A$. So we can find $s_n$ in $GL(A/I)$ with $s_n\to \pi(t)$. Choose lifts $t_n$ of $s_n$ such that $t_n\to t$. Then each $\pi(t_n)$ is a liftable invertible. Also $t_n$ is $\ep_n$--almost regular, where $\ep_n\to 0$, by Proposition 1.1. Hence the hypothoses of Lemma 2.1 are satisfied and dist$(t_n, GL(A))\to 0$. It follows that $t\in {\overline{GL(A)}}$.
\end{proof}

\section{\textbf{Non--stable $K$--theory and $IR$}}

A thorough treatment of $K-$theory of $C^*-$algebras can be found in \cite{Bl}.  We provide here a very brief introduction for the benefit of the reader.

If $p$ and $q$ are projections in a $C^*$--algebra $A$, we say $p$ and $q$ are Murray--von Neumann equivalent, in symbols $p\sim q$, if there is $u$ in $A$ such that $u^*u=p$, $uu^*=q$. For a unital $C^*-$algebra $A$ let $V_n(A)$ be the set of Murray--von Neumann equivalence classes of projections in $A\otimes\bM_n$, where $\bM_n$ is the algebra of $n\times n$ matrices.  Then let $V(A)$ be the direct limit of the $V_n(A)$'s, where the map from $V_n(A)$ to $V_{n+1}(A)$ is given by $[p] \mapsto [\begin{pmatrix} p & 0\\0 & 0\end{pmatrix}]$.  Here we are identifying $A\otimes\bM_n$ with the algebra of $n\times n$ matrices over $A$.  Then $V(A)$ becomes an abelian semigroup under the operation $[p] + [q] = [\begin{pmatrix} p & 0\\0 &q\end{pmatrix}]$.  Then $K_0(A)$ is the Grothendieck group of $V(A)$.  This means that $K_0(A)$ is generated by $V(A)$ and that $[p]=[q]$ in $K_0(A)$ if and only if $[p] + [r]=[q] + [r]$ in $V(A)$ for some $[r]$.  It can be shown that it is enough to consider $r=\b1_n$ for sufficiently large $n$, where $\b1_n$ is the identity of $A\otimes\bM_n$.  If $A$ is non--unital, then $K_0(A)$ is the kernel of the natural map from $K_0(\tA)$ to $K_0(\bC)$.

For unital $A$ let $GL_n(A)=GL(A\otimes \bM_n)$.  Then $K_1(A)$ is the direct limit of the groups of homotopy classes in $GL_n(A)$ relative to  the maps $u\mapsto \begin{pmatrix} u & 0\\0 & \b1\end{pmatrix}$ from $GL_n$ to $GL_{n+1}$.  For $A$ non--unital $K_1(A)=K_1(\tA)$.  (Note that $K_1(\bC)=0$.)

One of the key features of $K-$theory is the exact sequence which exists whenever $I$ is an ideal of $A$.  This is a cyclic six term sequence which includes the natural maps from $K_*(I)$ to $K_*(A)$ and $K_*(A)$ to $K_*(A/I)$ as well as two boundary maps.  One boundary map goes from $K_1(A/I)$ to $K_0(I)$ and the other from $K_0(A/I)$ to $K_1(I)$.

It follows from $p\sim q$ that $p$ and $q$ generate the same ideal of $A$. The next theorem applies to arbitrary $C^*$--algebras with $IR$ and does not explicitly mention $K$--theory.

\begin{theorem} Let $p$ and $q$ be projections in a $C^*$--algebra $A$ with $IR$. If $p$ and $q$ generate the same ideal and if $(\b1-p)\sim(\b1-q)$ in $A^+$, then $p\sim q$.
\end{theorem}

\begin{proof} Choose $u$ in $A^+$ such that $u^*u=\b1-p$ and $uu^*=\b1-q$. Then $u\in R(A^+)$ and hence $u\in{\overline {GL(A^+)}}$. By R\o rdam \cite{R1}, this implies that there is a unitary $w$ in $A^+$ such that $w(\b1-p)=u$. Then $wp$ is a partial isometry in $A$ which implements the Murray--von Neumann equivalence between $p$ and $q$.
\end{proof}

The converse of this theorem is valid in the real rank zero case.

\begin{theorem} Let $A$ be a $C^*$--algebra of real rank zero. Then $A$ has $IR$ if and only if whenever $p$ and $q$ are projections in $A$ generating the same ideal, then $(\b1-p)\sim(\b1-q)$ in $A^+$ implies $p\sim q$. 
\end{theorem}

\begin{proof} For the not already proved direction, let $t\in (\b1+A)\cap R(A^+)$. Let $t=v|t|$ be the canonical polar decomposition in $(A^+)^{**}$. By \cite{R1}, to prove $t\in {\overline {GL(A^+)}}$, it is sufficient to show that for each $\ep$ with $0<\ep<1$, there is a unitary $u$ in $A^+$ such that $u\chi_{[\ep,\infty)}(|t|)=v\chi_{[\ep,\infty)}(|t|)$. Now by \cite{B} there is a projection $p'$ in $A^+$ such that $p'\ge\chi_{[\ep,\infty)} (|t|)$ and $p'\chi_{[0,\ep/2]}(|t|)=0$.
Let $s=tp'$. Then $s$ has closed range, and hence $s$ has a canonical polar decomposition, $s=w|s|$, with $w$ a partial isometry in $A^+$. Let $q'=ww^*$, $p=\b1-p'$, and $q=\b1-q'$. Then $p$ is the kernel projection of both $s$ and $w$, and $q$ is the cokernel projection of both $s$ and $w$. Note that $p$ and $q$ are in $A$, since $w\not\in A$. By Proposition 1.12 $s\in R(A^+)$. Therefore $p$ and $q$ generate the same ideal. By construction, $(\b1-p) \sim (\b1-q)$. So by hypothesis $p\sim q$. If $w_0^*w_0=p$ and $w_0w_0^*=q$, then $u=w_0+w$ is a unitary and $u\chi_{[\ep,\infty)}(|t|)=up'\chi_{[0,\ep)}(|t|)=w\chi_{[0,\ep)}(|t|)=v\chi_{[0,\ep)}(|t|)$. 
\end{proof}

Theorem 3.2 is a generalization of the previously noted fact that purely infinite simple $C^*$--algebras have $IR$. In fact Cuntz proved in \cite{C} that any two non-zero projections in a purely infinite simple $C^*$--algebra are equivalent if they have same image in $K_0$. Also purely infinite simple $C^*$--algebras have real rank zero by \cite{Z}.

A $C^*$--algebra $A$ will be said to have {\it stable $IR$} if $A\otimes\bM_n$ has $IR$ for all $n$. It follows from Theorem 2.5 that in the non-unital case, $\tA\otimes\bM_n$ has $IR$ (not just $(A\otimes\bM_n)^\sim$). By Proposition 1.9 stable $IR$ implies that $A\otimes\bK$ has $IR$, where $\bK$ is the $C^*$--algebra of compact operators on a separable infinite dimensional Hilbert space. Conversely if $A\otimes\bK$ has $IR$, then Proposition 1.4 implies $A$ has stable $IR$. Thus stable $IR$ is invariant under stable isomorphism. 
In \cite{BP2} a $C^*$--algebra $A$ was said to have {\it weak cancellation} if the following is true: Whenever $p$ and $q$ are projections in $A$ which generate the same ideal $I$ and which have the same image in $K_0(I)$, then $p\sim q$. Rieffel showed in \cite{Ri} that $C^*$--algebras of stable rank one satisfy a stronger cancellation property: If $p$ and $q$ have the same image in $K_0(A)$, then $p\sim q$. Of course $p\sim q$ implies that $p$ and $q$ generate the same ideal $I$ and have the same image in $K_0(I)$. We show next that stable $IR$ implies a cancellation property that is intermediate between these two properties.

\begin{theorem}
If $A$ is a $C^*$--algebra with stable $IR$ and if $p$ and $q$ are projections in $A$ which generate the same ideal and which have the same image in $K_0(A)$, then $p\sim q$. 
\end{theorem}

\begin{proof} Since $p$ and $q$ have the same image in $K_0(A)$, also $\b1-p$ and $\b1-q$ have the same image in $K_0(\tA)$. This implies that for sufficiently large $n$
\[\begin{pmatrix} \b1-p & 0\\
0 & \b1_{n-1}\end{pmatrix} \sim
\begin{pmatrix} \b1-q & 0\\
0 & \b1_{n-1}\end{pmatrix},\]
where  the equivalence takes place in $\tA\otimes\bM_n$. In other words if $p'=\begin{pmatrix} p & 0\\ 0 & 0_{n-1}\end{pmatrix}$ and $q'=\begin{pmatrix} q & 0\\ 0 & 0_{n-1}\end{pmatrix}$, then $(\b1_n-p')\sim (\b1_n-q')$. So Theorem 3.1 implies $p'\sim q'$, whence $p\sim q$. 
\end{proof}

\begin{proposition} If $A$ is a $C^*$--algebra with stable $IR$ and if $I$ is an ideal of $A$, then the natural map from $K_0(I)$ to $K_0(A)$ is injective.
\end{proposition}

\begin{proof} By the $K$--theory exact sequence it is sufficient to show the natural map from $K_1(A)$ to $K_1(A/I)$ is surjective. This follows from the fact that invertibles lift from $(\tA/I)\otimes \bM_n$ to $\tA\otimes\bM_n$, which in turn follows from Theorem 2.5.
\end{proof}

In \cite{BP2} the following terminology was introduced, though the concepts involved were not new. A $C^*$--algebra $A$ has {\it $K_1$--surjectivity} if the natural map from $GL(\tA)$ to $K_1(A)$ is surjective and {\it $K_1$--injectivity} if that map is injective on homotopy classes. Also a non-unital $C^*$--algebra $A$ has {\it good index theory} if the following is true: Whenever $A$ is imbedded as an ideal in a unital $C^*$--algebra $B$, any element of $GL(B/A)$ which has index 0 can be lifted to $GL(B)$. 
Here if $[u]_1$ is the image of $u$ in $K_1(B/A)$, the index of $u$ is the image of $[u]_1$ in $K_0(A)$ under the boundary map of the $K$--theory exact sequence. An equivalent way to state good index theory is that if $[u]_1$ can be lifted to a class $\alpha$ in $K_1(B)$, then $u$ can be lifted to $GL(B)$. When stated in this way, there is a stronger property which is natural to consider: 
Demand that $u$ can be lifted to $v$ in $GL(B)$ such that $[v]_1=\alpha$.

It is not hard to see that this stronger property is equivalent to good index theory plus $K_1$--surjectivity. In fact if $A$ has $K_1$--surjectivity and $v$ in $GL(B)$ lifts $u$, then $\alpha-[v]_1$ is in the image of $K_1(A)$. Therefore we can simply multiply $v$ by an appropriate element of $(\b1+A)  \cap GL(\tA)$ to obtain a lift $w$ with $[w_1]=\alpha$. Conversely if the stronger property is satisfied, we obtain $K_1$--surjectivity from the special case $u=\b1$. In my prior experience whenever good index theory could be proved, $K_1$--surjectivity could also be proved.
 In \cite{BP2}, Pedersen and I proved that any extremally rich $C^*$--algebra with weak cancellation has good index theory, $K_1$--surjectivity, and $K_1$--injectivity. Since stable $IR$ implies a cancellation property stronger than weak cancellation, and since $IR$ is equivalent to extremal richness for simple $C^*$--algebras (see Propostion 4.2 below), it might be hoped that stable $IR$ would at least imply both good index theory and $K_1$--surjectivity. However, I have been successful only in proving good index theory.

\begin{theorem} Stable $IR$ implies good index theory.
\end{theorem}

\begin{proof} Assume $A$ has stable $IR$, $A$ is an ideal in a unital $C^*$--algebra $B$, and $s$ is an element of $GL(B/A)$ with index 0. Let $\pi:B\to B/A$ and $\pi_n:B\otimes\bM_n\to (B/A)\otimes \bM_n$ be the quotient maps. Since $[s]_1$ lifts to $K_1(B)$, for sufficiently large $n$ there is $v$ in $GL_n(B)$ such that $\pi_n(v)$ is homotopic to $\begin{pmatrix} s & 0\\ 0 & \b1_{n-1}\end{pmatrix}$ in $GL_n(B/A)$. By the homotopy lifting property for $\pi_n: GL_n(B)\to GL_n(B/A)$, $\begin{pmatrix} s & 0\\ 0 & \b1_{n-1}\end{pmatrix}$ lifts to $GL_n(B)$. Now choose $0<\ep<1$ and use Lemma 2.4 to find an $\ep$--almost regular lift $t$ of $s$. Then clearly $t'=\begin{pmatrix} t & 0\\ 0 & \b1_{n-1}\end{pmatrix}$ is $\ep$--almost regular in $B\otimes \bM_n$. Since the image of $t'$ is a liftable invertible, Lemma 2.1 applies and dist$(t', GL_n(B))\le\ep$. Choose $\ep'$ with $\ep<\ep'<1$ and $w=\begin{pmatrix} a & b\\ c & d\end{pmatrix}$ in $GL_n(B)$ with $\|w-t'\|<\ep'$. Since $\|d-\b1_{n-1}\|<1$, $d$ is invertible and we may perform elementary row and column operations by multiplying $w$ on the left by 
$\begin{pmatrix} \b1 & -bd^{-1}\\ 0 & \b1_{n-1}\end{pmatrix}$ and on the right by $\begin{pmatrix} \b1 & 0\\ -d^{-1}c & \b1_{n-1}\end{pmatrix}$. In the upper left hand corner we obtain $a-bd^{-1}c\in GL(B)$, and $\|a-bd^{-1}c-t\|< \ep'+(\ep')^2 / (1-\ep')$. 
Since $\ep'$ can be arbitrarily small, we see that $s$ can be approximated arbitrarily well by a liftable invertible $u$. If $\|us^{-1}-1\|<1$, we deduce that $s$ is liftable.
\end{proof}

\begin{theorem} If $I$ is a closed two--sided ideal of a $C^*$--algebra $A$, and if both $I$ and $A/I$ have stable $IR$, then $A$ has stable $IR$ if and only if the natural map from $K_1(A)$ to $K_1(A/I)$ is surjective.
\end{theorem}

\begin{proof}  If $A$ has stable $IR$, then by Theorem 2.5 invertibles lift from $(\tA/I)\otimes\bM_n$ to $\tA\otimes\bM_n$, and hence the $K_1$--map is surjective. If the $K_1$--map is surjective, then since $I\otimes\bM_n$ has good index theory, we see that invertibles lift from $(\tA/I)\otimes\bM_n$. Then Theorem 2.5 implies that $A$ has stable $IR$.
\end{proof}

\section{\textbf{Additional results and questions}}

\begin{proposition} If $A$ is a simple $C^*$--algebra, then $A$ has $IR$ if and only if either $A$ has stable rank one or $A$ is purely infinite.
\end{proposition}

\begin{proof} For the direction not already proved, assume that $A$ has $IR$ and is not of stable rank one. We need to show that every non-zero hereditary $C^*$--subalgebra $B$ of $A$ contains a non-zero projection and that every non-zero projection is infinite. For the first, note that $B$ cannot have stable rank one, since $B$ is strongly Morita equivalent to $A$ and the stable rank one property is preserved by strong Morita equivalence. We may assume $B$ is not unital, and thus $\tB$ can be identified with $B+\bC\b1\subset \tA$. Then there must be $t\in\b1+B$ such that $t\not\in R(\tB)$ (since $B$ has $IR$ by Proposition 1.4). Since the only relevant quotient of $\tB$ is $\tB$ itself, $t$ is one-sided invertible but not invertible, and $\tB$ contains a proper isometry $u$. Then $\b1-uu^*$ is the desired non-zero projection in $B$. For the second, let $p$ be a non-zero projection in $A$, and let $B=pAp$. The same sort of reasoning as above shows that there is a proper isometry in $B$, whence $p$ is infinite.
\end{proof}

\begin{theorem} If $A$ is a $C^*$--algebra with $IR$ and if $A$ is not of stable rank one, then $A$ has two closed two--sided ideals $I$ and $J$ such that $I\subset J$ and $J/I$ is purely infinite simple.
\end{theorem}

\begin{proof} Since $A$ does not have stable rank one, $R(\tA)\ne\tA$, and hence some quotient $\tA/I_0$  has a proper isometry $u$. Let $p=\b1-uu^*$ and let $I_1\supset I_0$ be an ideal maximal with respect to the property that $p\not\in I_1/I_0$. (Since the image of $p$ in any quotient is either 0 or of norm 1, if $\{L_\alpha\}$ is an upward directed family of ideals not containing $p$, then $(\cup L_\alpha)^-$ does not contain $p$. So Zorn's Lemma can be applied.) Let $\op$ be the image of $p$ in $\tA/I_1$. Let $J_1$ be the ideal such that $J_1/I_1$ is the ideal generated by $\op$. Then $J_1/I_1$ is simple. There is an isometry $\ou$ in $\tA/I_1$ such that $\b1-\ou\ou^*=\op$. Let $\oq=\b1-(\ou^2)(\ou^2)^*$. Then 
$\begin{pmatrix} \oq & 0\\
0 & 0\end{pmatrix} \sim
\begin{pmatrix} \op & 0\\
0 & \op\end{pmatrix}$
in $(\tA/I_1)\otimes \bM_2$. It follows that $\op$ and $\oq$ generate the same ideal (namely $J_1$) and of course $\b1-\oq\sim \b1\sim \b1-\op$. So Theorem  3.1 implies $\op\sim \oq$, whence $\op$ and $\oq$ are infinite. Therefore $J_1/I_1$ is not of stable rank one, and hence it is purely infinite. Now if $I=I_1\cap A$ and $J=J_1\cap A$, then $J/I\cong J_1/I_1$, since $\tA/A$ is at most one--dimensional and $J_1/I_1$ is simple and infinite--dimensional.
\end{proof}

\begin{corollary} Any $GCR$ $C^*$--algebra which has $IR$ is of stable rank one.
\end{corollary}

\begin{proof} No purely infinite simple $C^*$--algebra is $GCR$.
\end{proof}

 Since $IR$ is equivalent to extremally rich for simple $C^*$--algebras, we give a couple of simple examples to show there is no implication in general. The Toeplitz algebra is extremally rich but does not have $IR$. The multiplier algebra of a non-unital finite matroid $C^*$--algebra has $IR$ but is not extremally rich.

Here are some questions that I would very much like to know the answers to:
\begin{itemize}
\item[1] Is $IR$ preserved by arbitrary direct limits?\\
\item[2] Does $IR$ imply stable $IR$? What if we also assume real rank zero?\\
\item[3] Does stable $IR$ imply $K_1$--surjectivity? What if we also assume real rank zero?
\end{itemize}

Question 1 has an affirmative answer if all the $C^*-$algebras in the direct limit system have real rank zero.  This can be shown with the help of Theorem 3.2.  In Question 3 I could also ask about $K_1$--injectivity, but in view of the discussion preceeding Theorem 3.5, the lack of an answer to the $K_1$--surjectivity question is more frustrating.  With regard to the second part of the question, note that it is unknown whether real rank zero, all by itself, implies $K_1-$surjectivity.

A less important question is whether stable $IR$ is preserved by strong Morita equivalence and not just by stable isomorphism. Of course strong Morita equivalence is equivalent to stable isomorphism for $\sigma$--unital $C^*$--algebras, in particular for separable $C^*$--algebras. A positive answer to question 1 would probably yield a positive answer to this question, via standard techniques.  In particular, this is so in the real rank zero case.

My association with Ron Douglas was very beneficial to and influential in my career.  In particular my interests in extensions of $C^*-$algebras, $K-$theory of $C^*-$algebras, and stable isomorphism arose out of this association.

\bibliographystyle{amsplain}

\end{document}